\documentclass[10pt]{article}

\usepackage{latexsym}
\usepackage{amsmath, amsthm}
\usepackage{amsfonts}

\newtheorem{theorem}{Theorem}
\newtheorem{prop}{Proposition}

\newtheorem{corollary}{Corollary}

\def\Bbb#1{{\mathbb#1}}
\def\RR{\Bbb{R}}
\def\NN{\Bbb{N}}

\def\EE{\Bbb{E}}

\def\XX{\Bbb{X}}

\def\eps{\varepsilon}

\def\simlim{N,L\to\infty\atop{N/L\to\rho}}

\begin{document}

\begin{center}
\vspace*{5mm}

{\bf \Large Thermodynamic Limit for the Invariant Measures in Supercritical Zero Range Processes.}\\
\vspace{6mm}
In\'{e}s Armend\'{a}riz\footnote[1]{IME, Universidade de S\~ao Paulo, S\~ao Paulo, Brazil. On leave from Universidad de San Andr\'{e}s, Vito Dumas 284, B1644BID, Victoria,  Argentina. E-mail:{\tt iarmendariz@udesa.edu.ar}},\
Michail Loulakis\footnote[2]{Department of Applied Mathematics, University of Crete, and Institute of Applied and Computational Mathematics, FORTH,  Crete.\ Knossos Avenue, 714 09 Heraklion Crete, Greece. E-mail: {\tt loulakis@tem.uoc.gr}}
\end{center}
\vspace{6mm}
\small
\hspace{1cm} \begin{minipage}[t]{4.9in}
ABSTRACT: We prove a strong form of the equivalence of ensembles for the invariant measures of zero range processes conditioned to a supercritical density of particles. It is known that in this case there is a single site that accomodates a macroscopically large number of the particles in the system. We show that in the thermodynamic limit the rest of the sites have joint distribution equal to the grand canonical measure at the critical density. This improves the result of Gro\ss kinsky, Sch\"{u}tz and Spohn, where convergence is obtained for the finite dimensional marginals. We obtain as corollaries limit theorems for the order statistics of the components and for the fluctuations of the bulk.
\end{minipage}
\normalsize
\vspace{1mm}
\section{Introduction}
In a landmark paper of 1970, Spitzer \cite{Sp} introduced five particle systems undergoing simple interactions and initiated a research project to rigorously analyse their equilibrium and dynamical properties. One of the systems he proposed was the zero range process, a model in which particles leave any given site at a rate $g(k)$ that only depends on the number $k$ of particles present at the site, hence the name. The attention was initially drawn to the existence of the dynamics under general conditions, the identification of invariant measures and the establishment of the hydrodynamic limit. All these questions have been successfully addressed, at least in the attractive case when the rate function $g(\cdot)$ is increasing. A comprehensive review of these results can be found in \cite{KL}.\\
\\  
Over the last decade, there has been an increasing interest in zero range processes such that the rate $g(\cdot)$  decreases with the number of particles. This can be thought as introducing a mechanism of effective attraction between the particles, that if strong enough, i.e. when the rates decrease sufficiently fast, can lead to phenomena of condensation -- a transition to a phase where a single site contains a finite fraction of the particles in the system. This type of condensation appears in diverse contexts such as traffic jamming, gelation in networks, or wealth condensation in macroeconomies, and zero range processes or simple variants have been used as prototype models. Evans and Hanney \cite{EH} provide an excellent review on this subject.\\
\\
A phase transition in this class of zero range processes
can be already observed at the level of the invariant states. It is known \cite{JMP,GSS} that when the density of particles exceeds a critical value $\rho_c$, the invariant measures of the process concentrate on configurations where a macroscopic proportion of the total number of particles forms a randomly located cluster. In this article we analyse the thermodynamic limit of the invariant measures of the process conditioned to having a supercritical density, that is we let the number of sites $L$ and the number of particles $N$ grow to infinity in such a way that $N/L\to\rho>\rho_c$.\\
\\
Given the particle and site numbers $N$ and $L$ as above, the invariant state of the process is identified as the product of $L$ copies of a measure $\nu_{\phi_c}$ supported on the integers, conditioned to adding up to $N$. When the particle density $N/L$  is higher than $\rho_c$ we are conditioning on an atypical event, and the problem can be described as Gibbs conditioning for  a measure having no exponential moments. Gro\ss kinsky, Sch\"{u}tz and Spohn \cite{GSS} identified the typical configuration of a finite subsystem by proving an equivalence of ensembles property. Remarkably, the effect of the conditioning on the finite subsystem disappears in the thermodynamic limit. This happens because the rare event is most likely realised by a large deviation of the maximum component. A similar result was proved by Ferrari, Landim and Sisko \cite{FLS} when the number of sites is fixed while the particle number grows to infinity, and by Gro\ss kinsky \cite{G} for systems with two particle species. \\ 
\\
The fact that convergence to a product measure holds for the finite dimensional marginals is standard when the equivalence of ensembles or the Gibbs conditioning principle are satisfied.  It is crucial that the size of the subsystem amounts to a vanishing fraction of the whole. Indeed, the result often fails to hold when this is not the case (cf. Proposition 2.12 in \cite{DZ}). The main result in this article is an unusually strong form of the equivalence of ensembles. Precisely, we prove (Theorem \ref{equivalence}) that in supercritical zero range processes the effect of conditioning is entirely absorbed by the maximum component, in the sense that the joint distribution of the remaining sites converges to a product measure. We then derive some interesting corollaries from this result.\\
\\
This distinctive behavior can be attributed to the fact that the marginals $\nu_{\phi_c}$ of the unconditional distribution are subexponential. Indeed, the proof of Theorem \ref{equivalence} relies on a Local Limit Theorem in the form of equation (\ref{LLT}), a result that requires little more than subexponentiality. 

\section{Notation and results}
Zero range processes are interacting particle systems evolving on a set of sites $\Lambda$. Particles
perform random walks on $\Lambda$ interacting only with particles sitting on the same site through the 
following rule: the rate at which a particle leaves a site depends on the number of particles at
that site. Given a function $g:\ \NN_{0}=\{0,1,2,\ldots\}\mapsto\RR_{+}$ 
and a transition probability $p(\cdot,\cdot)$ on $\Lambda\times\Lambda$, the dynamics of the
process can be described as follows. If there are $k$ particles at a site $x$, then independently of the
configuration on the other sites, a particle leaves $x$ after an exponential waiting time with rate $g(k)$. 
A target site is chosen according to $p(x,\cdot)$, the particle jumps there and the process starts afresh.\\
\\
A zero range process can be rigorously defined as a Markov process on the state space
$
\XX_{\Lambda}=\NN_{0}^{\Lambda}.
$
A point $\eta$ in $\XX_{\Lambda}$ can be thought of as a configuration of particles on $\Lambda$, with 
$\eta_x$ denoting the number of particles at the site $x\in\Lambda$.
Regarding the jump rate function $g(\cdot)$ and the transition probabilities $p(\cdot,\cdot)$, we assume that
\[
g:\NN_{0}\mapsto \RR_{+} \text{ is such that } g(k)=0 \Leftrightarrow k=0,
\]
and
\[
p:\Lambda\times\Lambda\mapsto [0,1] \text{ is such that }\sum_{y\in\Lambda} p(x,y)=\sum_{y\in\Lambda} p(y,x)=1, \ \forall x\in\Lambda.
\]
In order to avoid degeneracies we further assume that the random walk on $\Lambda$ with transition probabilities
$p(\cdot,\cdot)$ is irreducible.  In this article we only consider finite sets $\Lambda$,
in which case we can define a process starting from any initial configuration $\eta\in\XX_{\Lambda}$.\\
\\
The infinitesimal generator of the zero range process is then given by
\[ 
Lf(\eta)=\sum_{x,y\in\Lambda}g\big(\eta_x\big)p(x,y)\left(f(\eta^{x,x+y})-f(\eta)\right),
\]
where
\[
\eta^{x,x+y}_z=\begin{cases}
                                \eta_z &\text{if }  z\neq x,y\\
                                \eta_x-1 &\text{if }  z=x \\
                                \eta_y+1 &\text{if }  z=y.
                           \end{cases}
\]
\\
\noindent
Zero range processes possess a family of invariant product measures with site marginals given by
\[
\nu_{\phi}\big[\eta_x=k\big]=\frac{1}{Z(\phi)}\frac{\phi^{k}}{g(k)!}\,,
\]
where $g(k)!=\prod_{m=1}^{k}g(m)$. Each of these measures is usually referred to as the grand-canonical ensemble corresponding to
the fugacity $\phi$, and they can be defined for any $\phi$ in the range of convergence of the power series
\[
Z(\phi)=\sum_{k}\frac{\phi^{k}}{g(k)!}.
\]
The expected number of particles per site is given by
\[
\rho(\phi)=\EE^{\nu_{\phi}}\big[\eta_x\big]=\frac{1}{Z(\phi)}\sum_{k=1}^{\infty}k\frac{\phi^{k}}{g(k)!}.
\]
It can be easily verified that $\rho$ is a strictly increasing function of $\phi$.\\
\\
Let $\phi_{c}\le+\infty$ denote the radius of convergence of $Z(\phi)$. If $Z(\phi_{c}):=\lim_{\phi\uparrow\phi_{c}}Z(\phi)=\infty$,
it can be proved \cite{KL} that $\rho_{c}:=\lim_{\phi\uparrow\phi_{c}}\rho(\phi)=\infty$.
If on the other hand $Z(\phi_{c})$ is finite, it is possible that $\rho_{c}$ is also finite. In this case none of the grand-canonical measures 
corresponds to a particle density higher than the critical $\rho_{c}$, and the system undergoes a phase transition \cite{JMP,GSS} from a fluid to
a condensed phase, in a sense to be made precise later. \\
\\
To fix ideas, we consider here a reference model such that both $Z(\phi_c)$ and $\rho_c$ are finite that was originally proposed by Evans \cite{E}. In the last section we discuss how
our results apply to a number of other systems with finite critical density. \\
\\
 In Evans' model the jump rates are given by
\begin{equation}
g(k)=\begin{cases}1+\frac{b}{k} &\text{ if } k\ge 1\\
                            0   &\text{ if } k=0.
       \end{cases}
\label{rates}
\end{equation}
With this choice of $g$, one gets
\[
g(k)!=\frac{\Gamma(b+k+1)}{\Gamma(b+1)k!}\sim \frac{k^{b}}{\Gamma(b+1)},
\]
if $\Gamma(\cdot)$ denotes the standard Gamma function. The critical fugacity $\phi_{c}$ is equal to 1, the partition function 
$Z(\phi)$ is finite at $\phi_{c}$ if $b>1$, and the critical density $\rho_{c}$ is finite if $b>2$. Since we are interested in systems with
finite critical density we will assume throughout this article that $b>2$. \\
\\
Due to the conservation of the number of particles by the dynamics, the state space is partitioned into 
finite invariant subspaces, where 
\[
S_{L}(\eta)=\sum_{x\in\Lambda}\eta_x
\]
is constant: $\XX_{\Lambda,N}=\{\eta \in \XX_{\Lambda}: S_{L}(\eta)=N\}.$
On each of these subspaces the zero range process is irreducible and has a unique invariant measure which
we denote by $\mu^{N,L}$. We will refer to the measures $\mu^{N,L}$ as the canonical ensembles. They can be explicitly computed, 
but they can also be obtained by conditioning the grand-canonical ensembles on the total number of particles. That is
\[
\mu^{N,L}\big[\cdot\big]=\nu_{\phi}^{L}\big[\cdot\ \big|\ S_{L}(\eta)=N\big].
\]
Note that the right hand side of the last equation does not actually depend on $\phi$. 
A natural object of interest is the behavior of these measures in the thermodynamic limit, as $N,L\to\infty$
in such a way that the average particle density $N/L$ converges to a constant $\rho$. \\
\\
When $\rho<\rho_c$ there exists a fugacity $\phi$ such that $\rho=\rho(\phi)$ and the standard equivalence of ensembles for independent random variables holds \cite{KL}. That is, the finite dimensional marginals of the canonical ensembles $\mu^{N,L}$ converge to  the grand-canonical ensemble corresponding to fugacity $\phi$. The equivalence of ensembles for (super)critical densities $(\rho\ge\rho_c)$ was established by Gro\ss kinsky, Sch\"utz and Spohn \cite{GSS}. Using relative entropy methods they prove convergence of the finite dimensional marginals of $\mu^{N,L}$ to the grand-canonical ensemble at critical fugacity.\\
\\
Furthermore, it has been proved \cite{JMP,GSS, G} that when the density is supercritical a condensation phenomenon emerges. Precisely, if $\rho>\rho_c$ and $\eps>0$ then
\begin{equation}
\lim_{\simlim}\mu^{N,L}\left[\frac{1}{L}\max_{x\in\Lambda}\eta_x>\rho-\rho_c-\eps\right]=1.
\label{weaklaw}
\end{equation}
This is to be contrasted with the size of the largest component in the case below criticality, which is of order $\log(L)$ \cite{JMP}. The comparison gives a precise meaning to the phase transition experienced by the system, and is reminiscent of the Erd\"os-Renyi results on the largest cluster of a random graph.\\
\\
The heuristic picture suggests that at supercritical densities the bulk of the sites is distributed according to independent copies of $\nu_{\phi_c}$, while a single randomly located site accumulates all the excess mass. The results mentioned above do not fully justify this picture however, because convergence to the grand-canonical ensembles is only obtained at the level of finite dimensional marginals. Hence, questions that require knowledge of the full limiting distribution of the bulk cannot be addressed directly. Such questions include for example the fluctuations of the bulk density around $\rho_c$,  the fluctuations of the maximum around $(\rho-\rho_c)L$, or the size of the second largest component.\\
\\
The contribution of this paper is a strong version of the equivalence of supercritical ensembles that provides a complete description for the thermodynamic limit and justifies the aforementioned picture.  
Precisely, if $\eta\in\XX_\Lambda$ is a configuration of particles on $\Lambda$ we define
\[
M_{L}(\eta)=\max_{x\in\Lambda}\eta_{x}
\]
and let $m_{L}(\eta)=\text{argmax}(\eta)$ be the position where the maximum occurs. We can always enumerate the sites of $\Lambda=\{x_1,\ldots,x_L\}$ and define $m_L(\eta)$ to be the site with the smallest index should the maximum occur more than once.\\
\\
We also define
\[
(\sigma^{y,z}\eta)_{x}=\begin{cases}\eta_{x} & \text{ if } x\neq y,z,\\
                                                                \eta_y & \text{ if } x=z,\\
                                                                \eta_z & \text{ if } x=y,
                                        \end{cases}
\]
\noindent
and the operator $T:\XX_{\Lambda}\longrightarrow\XX_{\Lambda}$ with $T\eta=\sigma^{x_L,m_L(\eta)}\eta$ that exchanges the last and the maximum component of $\eta$.\\
\\
We are ready to state the main result.
\begin{theorem} Let ${\cal F}_{L}$ be the $\sigma$-field generated by $\eta_{x_1},
\ldots\eta_{x_L}$. If $\rho>\rho_{c}$, then
\[
\lim_{\simlim}\sup_{A\in{\cal F}_{L-1}}\big|\mu^{N,L}\circ T^{-1}\big[A\big]-\nu_{\phi_{c}}^{L-1}\big[A\big]\big|= 0.
\]
\label{equivalence}
\end{theorem}
\noindent
This extends the result of Ferrari, Landim and Sisko \cite{FLS} to the case where the number of sites increases to infinity together with the number of particles, and that of Gro\ss kinsky, Sch\"{u}tz and Spohn \cite{GSS} in the sense that convergence to the grand canonical distribution is obtained
for the joint distribution under $\mu^{N,L}$ of {\em all} the components in the bulk.\\
\\
Given a measure $\mu$ defined on a $\sigma$--algebra ${\cal B}$, let $\|\cdot\|_{\text{t.v.}}$ stand for the total variation norm
\[
\|\mu\|_{\text{t.v.}}=\sup_{A \in {\cal B}}|\mu(A)|.
\]
It is not hard to see that Theorem \ref{equivalence} then implies that
\[
\Big\| \mu^{N,L}-\frac{1}{L}\sum_{x\in\Lambda} \nu^{N,L}\circ\sigma^{x,x_L}\Big\|_{\text{t.v.}}\to 0\,,
\]
where $\nu^{N,L}$ is a probability measure on $\XX_\Lambda$ with marginal on ${\cal F}_{L-1}$ given by
$\nu_{\phi_c}^{L-1}$, and such that the distribution of $\eta_{x_L}$ given ${\cal F}_{L-1}$ equals 
the Dirac measure at $N-\sum_{j=1}^{L-1}\eta_{x_j}$.\\
\\
Several interesting facts about the invariant measures of the zero range process at supercritical densities
are now simple consequences of Theorem \ref{equivalence}. In view of (\ref{weaklaw}) we would like to compute the fluctuations of $M_L(\eta)$ around $(\rho-\rho_c)L$. This question was raised already in \cite{JMP} and has been
numerically investigated by Godr\`{e}che and Luck (see appendix A.2.2 in \cite{GL}). The numerical experiments
suggest that for $b>3$ the fluctuations of $M_{L}$ are of order $\sqrt{L}$ and Gaussian, while for $2<b<3$ they are of order $L^{\frac{1}{b-1}}$. 
Theorem \ref{equivalence} and the obvious equality
\[
M_{L}(\eta)=N-\sum_{x=1}^{L-1}(T\eta)_{x},\qquad \mu^{N,L}-a.s,
\]
imply that the fluctuations of the maximum component reduce to the fluctuations of the sum of $L-1$ independent random variables with mean $\rho_{c}$ around $\rho_{c}(L-1)$, for which standard central limit theorems are available \cite{GK}. The precise result is the following:
\begin{corollary} Suppose $\rho>\rho_{c}$.\\
a) If $b>3$, that is if $\nu_{\phi_{c}}$ has finite variance $\sigma^{2}=\frac{(b-1)^{2}}{(b-2)^{2}(b-3)}$, then for all $x\in\RR$:
\[
\lim_{\simlim}\mu^{N,L}\left[\frac{M_{L}(\eta)-(N-\rho_{c}L)}{\sigma L^{1/2}}\le x\right]=\frac{1}{\sqrt{2\pi}}\int_{-\infty}^x e^{-u^{2}/2}du.
\]
b) If $b=3$, then for all $x\in\RR$:
\[
\lim_{\simlim}\mu^{N,L}\left[\frac{M_{L}(\eta)-(N-\rho_{c}L)}{2\sqrt{L\log L}}\le x\right]=\frac{1}{\sqrt{2\pi}}\int_{-\infty}^x e^{-u^{2}/2}du.
\]
c) If $2<b<3$, then for all $x\in\RR$:
\[
\lim_{\simlim}\mu^{N,L}\left[\frac{M_{L}(\eta)-(N-\rho_{c}L)}{\big(\Gamma(b)L\big)^{\frac{1}{b-1}}}\le x\right]=\int_{-\infty}^{x}{\cal L}_{b-1}(u)\,du.
\]
where ${\cal L}_{\alpha}$ is the density of the completely asymmetric stable law with index $\alpha$ and characteristic function $\psi(t)$ given by:
\[
\log\psi(t)=\int_{-\infty}^{0}\left(e^{itx}-1-itx\right)\frac{\alpha dx}{|x|^{\alpha+1}}=-C_{\alpha}|t|^{\alpha}\left(1+i\  \text{\em sgn}(t)\tan\frac{\pi\alpha}{2}\right)
\]
\label{scalinglaws}
\end{corollary}
\noindent
Note that for $b=3$ we still have Gaussian fluctuations after proper scaling. \\
\\
Clearly, one can go on and obtain limit theorems for the statistics of any order under $\mu^{N,L}$ from the corresponding result for product measures. For instance, the second largest component is given by 
\[
M_{L}^{(2)}(\eta)=\max_{1\le x\le L-1}(T\eta)_{x}
\]
and the following limit theorem is a direct consequence of Theorem \ref{equivalence} and the estimate (\ref{asymptotics}) for the tail probabilities under $\nu_{\phi_c}$.
\begin{corollary}
Suppose $b>2$ and let $\rho>\rho_{c}$. Then, for any $x>0$
\[
\lim_{\simlim}\mu^{N,L}\left[M_{L}^{(2)}(\eta)\le x{\Big(\Gamma(b)L\Big)^{\frac{1}{b-1}}}\right]\ =\ e^{-x^{1-b}}.
\]
\label{secondlargest}
\end{corollary}
\noindent
The fluctuations of the bulk are closely related to the fluctuations of the maximum component.
It follows from Corollary \ref{secondlargest} that in the limit, $m_L$ is the only site where the number of particles is of order $L$. Given $\zeta\in(0,\rho-\rho_c)$ we define the bulk configuration as $\eta_x^{*}=\eta_x 1_{\{\eta_x<\zeta L\}}$, and the rescaled bulk fluctuation process $Y_L(\cdot)\in D[0,1]$ as
\[
Y_L(t)=\frac{1}{a_L}\sum_{j=1}^{[Lt]} (\eta_{x_j}^{*}-\rho_c),
\]
where
\begin{equation}
a_L=\begin{cases} \sigma\sqrt{L} &\text{ if } b>3 \\
			      2\sqrt{L\log L} &\text{ if } b=3 \\
			      \big(\Gamma(b)L\big)^{\frac{1}{b-1}} &\text{ if } 2<b<3.
	\end{cases}
\label{aL}
\end{equation}
\\
\noindent
The following corollary follows easily from Theorem \ref{equivalence} and Donsker's invariance principle or
its extension by Skorokhod (Theorem 2.7 in \cite{S}) to i.i.d. random variables in the domain of attraction of a stable law.
\begin{corollary}
Suppose $\rho>\rho_c$ and let $b>2$. Then under $\mu^{N,L}$
\[
Y_{L}(\cdot)\stackrel{d}{\longrightarrow} \xi_b(\cdot),\qquad \text{as } N\to\infty,\ L\to\infty,\ N/L\to\rho , 
\]
where $\xi_b$ is a standard Wiener process if $b\ge 3$, or a completely asymmetric stable process 
with index $\alpha=b-1$ and characteristic exponent
\[
\log\psi(-t)=\int_{0}^{\infty}\left(e^{itx}-1-itx\right)\frac{\alpha dx}{|x|^{\alpha+1}}=-C_{\alpha}|t|^{\alpha}(1-i\  \text{\em sgn}(t)\tan\frac{\pi\alpha}{2}),
\]
if $2<b<3$.
\label{bulk}
\end{corollary}
\noindent
It is worth comparing Corollary \ref{bulk} with the bulk fluctuations at criticality. If  $N=[\rho_c L]$  then, according to a result in Thomas Liggett's dissertation (cf. Theorem 4 in \cite{L}),  $Y_L{(\cdot)}$ converges in distribution to the bridge of $\xi_b$ conditioned to return to the origin at time 1. \\
\\
\noindent
Theorem \ref{equivalence} can be also applied to the numerical simulation of the invariant states $\mu^{N,L}$, when $N/L\to\rho>\rho_{c}.$
For large $L$, instead of drawing a sample from a distribution $\mu^{N,L}$, it is computationally more efficient to draw $L-1$ independent samples from 
a distribution $\nu_{\phi_{c}}$, and assign the rest of the mass to a site uniformly distributed in $\{1,2,\ldots,L\}$. \\
\\
We present the proof to the main result in the following section. We conclude (section 4) by discussing two questions that arise naturally from Theorem \ref{equivalence}. In the first one, we study a model such that the associated invariant measure  $\nu_{\phi_c}$ has a stretched exponential tail, and prove that Theorem 
\ref{equivalence} still holds.  In the second one, we consider a family of  systems with particle numbers $N$ deviating moderately from the typical value $\rho_c L$,  and refine our estimate of the threshold of values for $N$ where a phase transition occurs.

\section{Proof of Theorem \ref{equivalence}}
We begin this section with a few observations on the model. Recall from the previous section that the jump rates are given by $g(k)=1+\frac{b}{k}$ for $k>0$, and the critical fugacity $\phi_{c}$ is equal to 1. Recall also that since we assume $b>2$ both $Z(\phi_c)$ and $\rho_{c}$ are finite. Although the precise value of the partition function, the critical density, or other statistics of $\nu_{\phi_{c}}$ are not important, it was pointed out in \cite{GSS} that they can be explicitly computed using the hypergeometric identity \cite{A}
\begin{equation}
\sum_{k=0}^{\infty} \frac{\Gamma(u+k)\Gamma(v+k)}{\Gamma(w+k)\ k!}=\frac{\Gamma(u)\Gamma(v)\Gamma(w-u-v)}{\Gamma(w-u)\Gamma(w-v)},
\label{hg}
\end{equation}
valid for any $u,v,w>0$ with $w>u+v$. For instance,
\[
Z(\phi_{c})=\frac{b}{b-1},\  \rho_{c}=\frac{1}{b-2}, \text{  and if } b>3 \text{ then } \sigma^2=\frac{(b-1)^2}{(b-2)^2(b-3)}.
\]
We will next derive a smoothness estimate for the function
\[
W(k)=\nu_{\phi_c}\big[\eta_x=k\big]=\frac{1}{Z(\phi_c)g(k)!}=\frac{(b-1)\Gamma(b)k!}{\Gamma(k+b+1)}.
\]
It is clear that $W$ is decreasing, while from the elementary inequality
\begin{equation}
1+x\ge e^{\frac{x}{1+x}}\ \ \  x>-1\,,
\label{elementary}
\end{equation}
one can easily deduce that $W(k)k^b$ is increasing. Thus, for $k_1\le k_2$ we get
\begin{equation}
W(k_1)\ge W(k_2)\ge W(k_1)\left(\frac{k_1}{k_2}\right)^b.
\label{smoothness}
\end{equation}
We can also apply (\ref{hg}) to compute the tail probabilities of $\nu_{\phi_c}$ as follows
\[
\sum_{k=m}^\infty W(k) 
=(b-1)\Gamma(b)\sum_{k=0}^{\infty}\frac{\Gamma(m+1+k)}{\Gamma(m+b+1+k)}
=\frac{\Gamma(b)\ m!}{\Gamma(m+b)}.
\]
Hence, if we denote by $F$ the distribution function of $\nu_{\phi_{c}}$
and by $\bar{F}=1-F$ its tail, we get the following asymptotic behavior at infinity
\begin{equation}
W(k)\sim (b-1)\Gamma(b)k^{-b},\ \text{ and }\ \bar{F}(x)\sim\Gamma(b)x^{1-b}.
\label{asymptotics}
\end{equation}
\noindent
This observation explains the normalizing constants in the statements of  Corollary \ref{secondlargest} and Corollaries \ref{scalinglaws} and \ref{bulk} for $2<b<3$. The logarithmic correction when $b=3$ comes from the direct computation
\[
\EE^{\nu_{\phi_c}}\big[\eta_x^2\ 1_{\{\eta_x\le L\}}\big]\sim 4\sum_{k=1}^{L}\frac{1}{k}\sim 4\log L.
\]
\noindent
The proof of Theorem \ref{equivalence} relies on a local limit theorem for the (unconditioned) measure at criticality. It estimates the probability of the event we are conditioning upon in the definition of 
$\mu^{N,L}$. Such a result first appeared in Nagaev \cite{N1} for $b>3$ and Tka\v{c}uk \cite{T} for 
$b<3$. Baltrunas \cite{B} gives an accessible proof that encompasses all values of  $b>2$.
\begin{prop} If $\rho>\rho_{c}$ then
\begin{equation}
\lim_{\simlim}\frac{\nu_{\phi_{c}}^L\big[S_{L}(\eta)=N\big]}{L\nu_{\phi_{c}}\big[\eta_{x}=N-[\rho_{c}L]\big]}=1.
\label{LLT}
\end{equation}
\label{locallimit}
\end{prop}
\noindent Equation (\ref{LLT}) says that the most probable way that the rare event $\{S_{L}(\eta)=N\}$ occurs is when one variable takes up all the "excess mass", while the remaining $L-1$ ones assume typical values. This behaviour is to be contrasted with the large deviations behavior for random variables with finite exponential moments, where the rare event is realised by all variables taking values close to the atypical $\rho$. \\
\noindent
We proceed now with the proof of Theorem \ref{equivalence}.
\begin{proof}[Proof of Theorem \ref{equivalence}] Recall from Section 2 that $\sigma^{x_i,x_j}$ stands for the mapping that exchanges the $i$--th and the $j$--th  components of $\eta$, and that $T$ denotes the transformation that exchanges the last and the maximum components of $\eta$.\\

\noindent Let $A\subseteq \{\eta:\ \eta_{x_L}>\eta_{x_j},\ j=1,2,\ldots,L-1\}$. Due to the invariance of $\mu^{N,L}$ under $\sigma^{x_L,x_{\ell}},\,\ell=1,\dots,L$, we get
\begin{align}
\mu^{N,L}\big[T^{-1}A\big] &= \sum_{\ell=1}^{L}\mu^{N,L}\big[T^{-1}A\cap \{m_{L}=\ell\}\big]\nonumber\\
&=  \sum_{\ell=1}^{L}\mu^{N,L}\circ\sigma^{x_{L},x_{\ell}}\big[A]\,=\,L\mu^{N,L}\big[A]\nonumber\\
&= L\,\frac{\nu_{\phi_c}^L\big[A\cap\{S_L(\eta)=N\}\big]}{\nu_{\phi_c}^L\big[S_L(\eta)=N\big]}. 
\label{exchange}
\end{align}

\noindent
Consider a sequence $C_L$ such that $C_L/L\to 0$ and $C_L/a_L\to\infty$, where $a_L$ is defined in (\ref{aL}). Let  $D_L=\{m:\ |N-\rho_c L -m|<C_L\}$,  $t_L=N-\rho_c L -C_L$ and
$
B_L=\{\eta: \eta_{x_L}\in D_L;\ \max_{1\le j\le L-1} \eta_{x_j}\le t_L\}$.\\

\noindent Suppose now that $A\in{\cal F}_{L-1}=\sigma\{\eta_{x_1},\ldots,\eta_{x_{L-1}}\}$
We will apply equation (\ref{exchange}) to  the set $A\cap B_L$. Note that
\[
\nu_{\phi_c}^{L}\big[A\cap B_L\cap\{S_L(\eta)=N\}\big]=\sum_{m\in D_L} W(m)\nu_{\phi_c}^{L-1}\Big[ A\cap\big\{\sum_x \eta_x=N-m;\ \max_x \eta_x\le t_L\big\}\Big].
\]
\noindent
In view of (\ref{smoothness}) we can replace each value $W(m)$ in the range of summation by $W\big(N-[\rho_c L]\big)$, creating an error that is negligible as $L\to\infty$ uniformly in $A$. That is,
\begin{align*}
&\nu_{\phi_c}^L\big[A\cap B_L\cap\{S_L(\eta)=N\}\big]\\
&\hspace{2cm}= W\big(N-[\rho_c L]\big)\ \nu_{\phi_c}^{L-1}\bigg[ A\cap\Big\{\big|\sum_x \eta_x-\rho_c L \big| <C_L;\ \max_x \eta_x\le t_L\Big\}\bigg] \big(1+o(1)\big).
\end{align*}
\noindent
Since $C_L/a_L\to\infty$, the central limit theorem implies that 
\begin{equation*}
\nu_{\phi_c}^{L-1}\Big[\  \big|\sum_x \eta_x-\rho_c L \big| <C_L\Big]\longrightarrow 1\ \text{ as } L\to\infty,
\end{equation*}
\noindent
and there is also the elementary estimate
\begin{equation*}
\nu_{\phi_c}^{L-1}\big[ \max_x \eta_x\le t_L\big] = \big(1-\bar{F}(t_L)\big)^{L-1}\longrightarrow 1\ \text{ as } L\to\infty.
\end{equation*}
\noindent
Combining these two observations, we get
\[
\nu_{\phi_c}^L\big[A\cap B_L\cap\{S_L(\eta)=N\}\big]= W\big(N-[\rho_c L]\big)\ \nu_{\phi_c}^{L-1}\big[A\big]\  \Big(1+o(1)\Big).
\]
\noindent
Together with equation  (\ref{exchange}) and Proposition \ref{locallimit} this establishes that 
\[
\lim_{\simlim}\sup_{A\in{\cal F}_{L-1}}\Big|\mu^{N,L}\circ T^{-1}\big[A\cap B_L\big]-\nu_{\phi_c}^{L-1}\big[A\big] \Big|=0.
\] 
\noindent
In particular, if $A=\XX_{\Lambda}$, we get that 
\[
\lim_{\simlim}\mu^{N,L}\circ T^{-1}\big[B_L^{c}\big]=0,
\]
\noindent
$B_L^c=\XX_{\Lambda}\setminus B_L$. The assertion of the Theorem now follows by combining the last two equations.
\end{proof}

\section{Remarks}
We identified the condensation phenomenon present in supercritical zero range processes by proving the equivalence of ensembles in the standard Evans' model. It should be clear however that the essential ingredient for the proof is a Local Limit Theorem in the form of (\ref{LLT}). There are thus two possible directions to generalise Theorem \ref{equivalence}.
Its validity should be established for a greater variety of models, and the point where the phase transition with the emergence of a large cluster occurs should be determined with greater accuracy.\\
\\
We describe next how the proof can be adapted to a model for condensation with stretched exponential tails, also proposed by Evans.\\
\\
Suppose the jump rates are given by the function $g$ with
\begin{equation}
g(k)=\begin{cases} 1+\frac{\beta}{k^{\lambda}} & \text{ if } k>0\\     
                                   0 & \text{ if } k=0,
                                   \end{cases}
\label{strexp}
\end{equation}
where $\lambda\in(\frac{1}{2},1).$
The critical fugacity is still 1, although it is not possible to explicitly compute the distribution function and the critical density in this case. Nevertheless, it is elementary to see that $W(k)=\nu_{\phi_c}\big[\eta_x=k\big]$ is decreasing while $W(k)\exp\big(\frac{\beta k^{1-\lambda}}{1-\lambda}\big)$ is increasing so that we have
\begin{equation}
W(k_1)\ge W(k_2)\ge W(k_1)\exp\Big(-\beta\,\frac{k_2^{1-\lambda}-k_1^{1-\lambda}}{1-\lambda}\Big),\hspace{1cm} k_1\le k_2.
\label{strsmoothness}
\end{equation}
In fact, using (\ref{elementary}) one can check that
\[
W(k)\le\exp\Bigg(-\sum_{m=1}^k \frac{\beta}{\beta+m^{\lambda}}\Bigg)\le C\exp\Big(-\frac{\beta \,k^{1-\lambda}}{1-\lambda}\Big),
\]
and  the following asymptotic behavior for $W$ holds
\[
W(k)\sim A\, \exp\Big(-\frac{\beta\, k^{1-\lambda}}{1-\lambda}\Big) \hspace{.5cm}\text{ as } k\to\infty.
\]
This yields the asymptotic behavior of $\bar{F}(x)$
\begin{equation}
\bar{F}(x)\sim \frac{Ax^{\lambda}}{\beta}\exp\Big(-\frac{\beta\, x^{1-\lambda}}{1-\lambda}\Big)\hspace{.5cm} \text{ as } x\to\infty.
\label{strtail}
\end{equation}
In this context, Nagaev \cite{N2} has proved that (\ref{LLT}) is satisfied as long as $N=\rho_c L+\gamma(L)L^{\frac{1}{2\lambda}}$ with $\gamma(L)\to\infty$ as $L\to\infty$. In view of equations (\ref{strsmoothness}) and (\ref{strtail}), we may choose the sequence  $C_L=\sqrt{L\log L}$ in the line following the expression (\ref{exchange}),  and adapt  the arguments presented in the previous section to prove the following theorem.\\
\\
{\bf Theorem 1a} {\em If} $g(\cdot)$ {\em is given by (\ref{strexp}) and} $N=\rho_c L+\gamma(L)L^{\frac{1}{2\lambda}}$ {\em where} $\lim\gamma(L)=\infty$, {\em then}
\[
\lim_{L\to\infty} \sup_{A\in{\cal F}_{L-1}}\big|\mu^{N,L}\circ T^{-1}\big[A\big]-\nu_{\phi_{c}}^{L-1}\big[A\big]\big|= 0.
\]
In a similar fashion we can relax the conditions on $N$ in Theorem \ref{equivalence} provided we prove the validity of (\ref{LLT}) for values of $N$ deviating only moderately from its typical value. For instance, when $b>3$  Theorem 2 in \cite{D} implies that if $(N-\rho_c L)/\sqrt{L}\to\infty$ then
\[
\nu_{\phi_c}^L\big[S_L(\eta)=N\big]=\frac{1}{\sigma\sqrt{L}}\,\varphi\left(\frac{N-\rho_c L}{\sigma\sqrt{L}}\right) \big(1+o(1)\big)\,+\,LW\big(N-[\rho_c L]\big)\big(1+o(1)\big),
\]
where $\varphi(\cdot)$ is the density of the standard normal distribution. It is not hard to see that in this case 
(\ref{LLT}) holds as long as
\begin{equation}
N=\rho_c L+\frac{b-1}{b-2}\sqrt{L\log L}\left(1+\frac{b}{2(b-3)}\frac{\log\log L}{\log L}+\frac{\gamma(L)}{\log L}\right), \text{ with } \lim_{L\to\infty}\gamma(L)=\infty.
\label{refin_range}
\end{equation}
Once again, choosing $C_L=\sqrt{L\log L}$ we can prove  the following refinement of Theorem \ref{equivalence}.\\
\\
{\bf Theorem 1b} {\em If} $g(\cdot)$ {\em is given by (\ref{rates}) with } $b>3$ {\em and} $N$ {\em is as in 
(\ref{refin_range}), then}
\[
\lim_{L\to\infty} \sup_{A\in{\cal F}_{L-1}}\big|\mu^{N,L}\circ T^{-1}\big[A\big]-\nu_{\phi_{c}}^{L-1}\big[A\big]\big|= 0.
\]
Similar refinements of Theorem \ref{equivalence} can be obtained for the case when $b\le 3$.\\
\\
\noindent
{\bf Acknowledgments:} We would like to thank Claudio Landim for suggesting a problem that eventually led to the current form of the article, and for useful conversations while this paper was being prepared. ML has been supported by a Marie Curie Fellowship  of the European Community Programme ``Improving Human Potential'' under the contract number MERG-CT-2005-016163. IA has been supported by FAPESP Grant No.2007/50230--1.\\

\end{document}